\documentclass{amsart}

\usepackage[dvipsnames]{xcolor}
\usepackage{tikz,pgf}
\usetikzlibrary{shapes,decorations,arrows,calc,arrows.meta,fit,positioning}
\usetikzlibrary{decorations.markings}
\usetikzlibrary{arrows, automata}
\usepackage{amsfonts,amssymb}
\usepackage{float}
\usepackage{hyperref}
\usepackage{enumerate}
\usepackage{amsmath, amsthm, amsfonts, mathabx}
\usepackage[capitalise]{cleveref}
\usepackage[toc]{appendix}

\usepackage[official]{eurosym}

\usepackage{nomencl}
\usepackage{bm}
\usepackage{amsfonts}
\usepackage{fancyhdr}
\usepackage{amscd}
\usepackage[english]{babel}
\usepackage{xcolor}
\usepackage[capitalise]{cleveref}
\usepackage{mathtools}
\newtheorem{thm}{Theorem}[section]
\newtheorem{cor}[thm]{\scshape{Corollary}}
\newtheorem{prop}[thm]{\scshape{Proposition}}
\newtheorem{lemma}[thm]{\scshape{Lemma}}
\newtheorem{claim}[thm]{\scshape{Claim}}
\newtheorem{question}[thm]{\scshape{Question}}

\newtheorem{remark}[thm]{\scshape{Remark}}

\newtheorem{defn}[thm]{\scshape{Definition}}

\newtheorem*{thm*}{Theorem}
\newtheorem*{cor*}{Corollary}

\theoremstyle{definition}

\def\={\cong}

\def\mc{\mathcal}
\def\mbb{\mathbb}

\newcommand{\subg}[1]{\langle #1 \rangle}
\newcommand{\enum}[2]{\begin{enumerate}[\hspace*{0.5cm}#1] #2 \end{enumerate}}

\newcommand{\Z}[0]{\mathbb{Z}}
\newcommand{\N}[0]{\mathbb{N}}

\usepackage{verbatim}
\usepackage{amsmath}  \usepackage{scalefnt}

\title[Simple groups with infinite verbal width and trivial positive theory]{Simple groups with infinite verbal width and the same positive theory as free groups}

\author[M. Casals-Ruiz]{Montserrat Casals-Ruiz}
\address{Ikerbasque - Basque Foundation for Science and Matematika Saila,  UPV/EHU,  Sarriena s/n, 48940, Leioa - Bizkaia, Spain}
\email{montsecasals@gmail.com}

\author[A. Garreta]{Albert Garreta}
\address{Matematika Saila,  UPV/EHU,  Sarriena s/n, 48940, Leioa - Bizkaia, Spain}
\email{albert.garreta.fontelles@gmail.com}

\author[I. Kazachkov]{Ilya Kazachkov}
\address{Ikerbasque - Basque Foundation for Science and Matematika Saila,  UPV/EHU,  Sarriena s/n, 48940, Leioa - Bizkaia, Spain}
\email{ilya.kazachkov@gmail.com}

\author[J. de la Nuez]{Javier de la Nuez Gonz\'{a}lez}
\address{Matematika Saila,  UPV/EHU,  Sarriena s/n, 48940, Leioa - Bizkaia, Spain}
\email{jnuezgonzalez@gmail.com}
\thanks{The authors are supported by ERC grant PCG-336983, Basque Government Grant IT974-16 and Spanish Government grant MTM2017-86802-P}
\begin{document}
\tikzset{middlearrow/.style={
decoration={markings,
mark= at position 0.5 with {\arrow{#1}} ,
},
postaction={decorate}
}
}

\maketitle

\begin{abstract}
In this paper we show that there exists an uncountable family of finitely generated simple groups with the same positive theory as any non-abelian free group. In particular, these simple groups have infinite $w$-verbal width for all non-trivial words $w$.
\end{abstract}
\section{Introduction}
  
  The notion of a word is fundamental in group theory. All combinatorial group theory is based on it: groups are described in terms of their presentations, elements of groups are represented by words, all algorithmic problems in groups are formulated in terms of words etc. Words can be used to describe prominent classes of groups, such as varieties, e.g. varieties of abelian, nilpotent, solvable groups; and more generally to study word equations and subgroups generated by their solutions, hence the notion of verbal subgroup.

  More precisely, a \emph{word} $w$ is simply an element of a free group $F(x_1, \dots, x_n)$. Given a group $G$, any word $w$ naturally defines a map $\varphi_w: G^n \to G$ as follows
  $$
  (g_1, \dots, g_n) \mapsto w(g_1, \dots, g_n)
  $$
  and a subgroup $w(G)< G$ generated by the image of $\varphi_w$, which is called the \emph{verbal subgroup of $G$} defined by $w$. The \emph{variety} defined by a word $w$ is the class of groups $G$ for which $w(G)$ is trivial. We will call a word $w$ \emph{trivial} if the verbal subgroup $w(F)=F$, where $F$ is a non-abelian free group.

  By definition, every element of the verbal subgroup $w(G)$ is a product of some values of the function $w$. Hence, a natural question to ask is, what is the smallest number $n$ such that every element of $w(G)$ can be represented as a product of at most $n$ values of $w$? If such a number $n$ exists, then the word $w$ is said to be of \emph{finite width} in $G$ and  $G$ is said to be \emph{$w$-elliptic}. If $G$ is a finite group, then obviously every word $w$ has finite width. However, we may ask whether or not the width of $w$ is bounded in some class of finite groups (e.g. in all finite groups, in all finite simple groups, in all finite $p$-groups, etc). If $C$ is a class of groups, then a word $w$ has \emph{uniform width} in $C$ if there exists a function $f(d)$ such that the width of $w$ in $G \in C$ is not more than $f(d)$, where $d$ is the smallest number of generators of $G$.

  There is a deep relation between the uniform verbal width of a class of finite groups $C$ and the topology of the corresponding pro-$C$ group, since verbal subgroups of uniform finite width are closed in profinite topology. This was the key step in the proof of Serre for pro-p groups and Nikolov-Segal for the profinite case that every finite index subgroup of a finitely generated pro-p/profinite group is open, see \cite{NS}.

  Not surprisingly, many classical problems in group theory revolve around laws, verbal subgroups and their width. A prominent example is Ore's Conjecture. In 1951 Ore conjectured that every element of every finite non-abelian simple group is a commutator. Much work on this conjecture has been done over the years until its proof has been completed by Liebeck, O'Brien, Shalev and Tiep in \cite{LOST}.
  
  In light of solution of Ore's conjecture, it is natural to ask whether or not there are infinite simple groups with infinite commutator width. In this direction, Barge and Ghys \cite[Theorem 4.3]{BG} showed that there are simple groups of symplectic diffeomorphisms of $\mathbb{R}^{2n}$  which possess nontrivial homogeneous quasi-morphisms, and thus their commutator width is infinite. Further, existence of finitely generated simple groups of infinite commutator width was proved in \cite{Muranov}. Recently, Caprace and Fujiwara proved that there are finitely presented simple groups for which the space of homogeneous quasi-morphisms is infinite-dimensional, and in particular whose commutator width is infinite \cite{CF}.
  
  Generalising Ore's conjecture, one is naturally brought to study its analogues for other words. A group $G$ is said to be verbally elliptic if it is $w$-elliptic for every word $w$ (this terminology was introduced by P. Hall). In this terminology, in \cite{LOST} the authors not only establish Ore's conjecture but prove a more general result, which we can qualitatively summarized as follows.
  \begin{thm}[Liebeck, O'Brien, Shalev and Tiep, \cite{LOST}] \label{thm:lost}
  The class of finite simple groups is uniformly verbally elliptic.
  \end{thm}
  
  As in the case with Ore's conjecture, in light of Theorem \ref{thm:lost}, it is natural to ask if there are simple groups which are \emph{verbally parabolic}, that is, simple groups in which every verbal subgroup has infinite width.
  
  In this paper we address this question formulated in a more general setting from the model-theoretic perspective: does there exist a finitely generated simple group with trivial positive theory?

  It is easy to see that if a group $G$ satisfies a positive sentence (in the language of groups), then so does any quotient of $G$. On the other hand, all (non-abelian) free groups satisfy the same collection of positive sentences. Hence, we say that $G$ has trivial positive theory if it satisfies the the same positive sentences as the free group.
  
  An important observation for us is that given a word $w$, the property that the $w$-verbal subgroup has finite width $k$ can be expressed by a non-trivial positive sentence.

  The main result of this paper is the following.
  
  \begin{thm}\label{t: main}
  There exist an uncountable family of  finitely generated simple groups with trivial positive theory.
  In particular, these groups have infinite $w$-verbal width for all non-trivial words $w$ and so are verbally parabolic.
  \end{thm}

  Our proof has two main ingredients. On the one hand, we generalise Camm's construction of an uncountable family of simple groups, \cite{Camm}. These groups have a very simple structure as amalgamated products of two free groups of rank $3$ over an infinite index subgroup. On the other hand, we show that the action of these groups on the associated Bass-Serre is weakly acylindrical and conclude from the authors' previous work, see \cite{CGN1}, that this family of groups has trivial positive theory and infinite dimensional second bounded cohomology.

  Our examples are not finitely presented, therefore, it is natural to ask the following questions.
  
  \begin{question}
  \begin{enumerate}\
  \item  Are there finitely presented simple groups with trivial positive theory?
  \item Is the positive theory of finitely presented simple groups with infinite commutator width constructed by Caprace and Fujiwara trivial?
  \item A well-known family of finitely presented simple groups was discovered by Burger and Mozes, \cite{BM}. Is the positive theory of Burger-Mozes groups trivial?
  \end{enumerate}
  \end{question}

  \newcommand{\so}[0]{\mathcal{S}}
  \newcommand{\bso}[0]{\tilde{\so}}
  \newcommand{\bih}[0]{\tilde{H}}

  \newcommand{\bv}[0]{\tilde{v}}
  \newcommand{\bw}[0]{\tilde{w}}
  \newcommand{\bx}[0]{\tilde{x}}
  \newcommand{\by}[0]{\tilde{y}}
  \newcommand{\bz}[0]{\tilde{z}}
  \newcommand{\bs}[0]{\tilde{\mathcal{S}}}
  \newcommand{\nat}[0]{\in\N}
  \newcommand{\ent}[0]{\in\Z}
  \newcommand{\zo}[0]{\Z^{*}}
  \newcommand{\izo}[0]{\in\zo}

  \newcommand{\V}[0]{\mathcal{V}}
  \newcommand{\W}[0]{\mathcal{W}}
  
  \renewcommand{\bf}[0]{\tilde{f}}
  \newcommand{\bg}[0]{\tilde{g}}
  \newcommand{\bh}[0]{\tilde{h}}

  \newcommand{\bt}[0]{\tilde{t}}
  \newcommand{\biv}[0]{\tilde{\V}}
  \newcommand{\biw}[0]{\tilde{\W}}
  
  \newcommand{\bro}[0]{\tilde{\rho}}
  \newcommand{\btheta}[0]{\tilde{\theta}}
  
  \newcommand{\pb}[0]{\mathcal{P}}    
  \newcommand{\pbw}[0]{\pb^{<\omega}} 
  \newcommand{\pis}[0]{\mathcal{I}}   
  
  \newcommand{\bda}[0]{\tilde{\delta}}
  \newcommand{\bd}[0]{\tilde{d}}
  
  \newcommand{\tuple}[0]{(\rho,\theta,\tilde{\rho},\tilde{\theta},\iota)}
  \newcommand{\otuple}[0]{\omega=(\rho,\theta,\tilde{\rho},\tilde{\theta},\iota)}
  \newcommand{\altxt}[3]{\begin{align}\tag{#1}\label{#2}\text{#3}\end{align}}
  \newcommand{\da}[0]{$\dagger$}
  \newcommand{\lbeq}[1]{\,\displaystyle_{=}^{#1}\,}

\section{An uncountable class of groups with trivial positive theory}
  
  In this section we introduce an uncountable family of groups parametrized by $\omega=(\rho, \theta, \tilde\rho, \tilde \theta, \iota)$, where $\rho, \theta, \tilde\rho, \tilde \theta$ are permutations of $\Z^*=\Z\setminus\{0\}$ and $\iota$ is a bijection between some sets. These groups are defined as amalgamated products of free groups of rank 3 over an infinite index subgroup, that is $G=F_3 \ast _{H(\rho,\theta)=\iota(\tilde{H}(\tilde \rho,\tilde\theta))}F_3$. We then show that this class of groups has trivial positive theory.
  
  \bigskip
  
  Let us start defining the class of groups. Consider two free groups $F=\subg{x,y,z}$ and $\tilde{F}=\subg{\bx,\by,\bz}$ and  two permutations $\rho$ and $\theta$ of the set $\mbb{Z}^*$.
  Let $H= H_{\rho, \theta}\leq F$ be the subgroup generated by the collection $\mathcal{S}_{\rho, \theta}=\V_{\rho,\theta }\cup\W_{\rho, \theta}$, where $\V_{\rho, \theta}=\{v_{j}^{l}\,|\,l\in\Z, j\in\zo\}$ and $\W_{\rho, \theta}=\{w_{j}^{l}\,|\,l\in\Z, j\in\zo\}$ and $v_{j}^l, w_j^{l}$ are defined by:
  \begin{align*}
  v^{l}_{j}=(x^{j}y^{-\rho(j)})^{z^{l}} \\
  w^{l}_{j}=(x^{j}zy^{-\theta(j)})^{z^{l}},
  \end{align*}
  see Figure \ref{fig:cover}. Unless strictly required, we will omit referring to $\rho$ and $\theta$ in the notation $\mc{S}_{\rho, \theta}$, and so on.

  \begin{figure}[h!]\label{fig:cover}
  \begin{tikzpicture}[scale=1.5]
  \tikzset{near start abs/.style={xshift=1cm}}

  \draw[middlearrow={latex}] (2,2) -- (2,3);
  \draw[middlearrow={latex}] (2,3) -- (2,4);
  \draw[middlearrow={latex}] (2,4) -- (2,5);
  \draw[middlearrow={latex}] (2,5) -- (2,6);
  \draw[->] (4,2) -- (4,1.3);

  \draw[color=green,middlearrow={>>}] (1,3) -- (2,3);
  \draw[color=green,middlearrow={>>}] (2,3) -- (3,3);
  \draw[color=green,middlearrow={>>}] (3,3) -- (4,3);
  \draw[color=green,middlearrow={>>}] (4,3) -- (5,3);
  \draw[color=green,middlearrow={>>}] (5,3) -- (6,3);
  
  \draw[color=green,middlearrow={>>}] (1,5) -- (2,5);
  \draw[color=green,middlearrow={>>}] (2,5) -- (3,5);
  \draw[color=green,middlearrow={>>}] (3,5) -- (4,5);
  \draw[color=green,middlearrow={>>}] (4,5) -- (5,5);
  \draw[color=green,middlearrow={>>}] (5,5) -- (6,5);

  \draw[scale=3,middlearrow={latex}] (1,1)  to[in=50,out=130,loop] (1,1);
  
  \draw[scale=3,middlearrow={latex}] (1,1.67)  to[in=50,out=130,loop] (1,1.67);
  
  \draw[scale=3,middlearrow={latex}] (1.34,0)  to[in=50,out=130,loop] (0,1.34);
  
  \draw[color=green,scale=3,middlearrow={>>}] (1.34,0)  to[in=140,out=220,loop] (0,1.34);
  
  \draw[color=red,scale=3,middlearrow={>}] (1.34,0)  to[in=-40,out=40,loop] (0,1.34);

  \draw[middlearrow={latex}] (5,3)  to[bend left=50] (6,3);
  \draw[middlearrow={latex}] (5,5)  to[bend left=50] (6,5);

  \draw[middlearrow={>}, color=red] (1,3)  to[bend left=50] (2,3);
  \draw[middlearrow={>}, color=red] (2,3)  to[bend right=50] (5,3);
  \draw[middlearrow={>}, color=red] (5,3)  to[bend right=50] (3,3);
  \draw[middlearrow={>}, color=red] (3,3)  to[bend left=50] (6,3);

  \draw[middlearrow={>}, color=red] (2,5)  to[bend right=50] (5,5);
  \draw[middlearrow={>}, color=red] (5,5)  to[bend right=50] (3,5);
  \draw[middlearrow={>}, color=red] (3,5)  to[bend left=50] (6,5);

  \draw[thick,dotted] (2,1.7)--(2,2);
  
  \filldraw (2,2) circle (1pt);
  \filldraw (2,3) circle (2pt);
  \filldraw (2,4) circle (1pt);
  \filldraw (2,5) circle (1pt);
  \filldraw (2,5) circle (1pt);
  \filldraw (2,6) circle (1pt);
  \filldraw (4,0) circle (2pt);

  \draw[thick,dotted] (2,6)--(2,6.3);
  
  \draw[thick,dotted] (0.7,3)--(1,3);
  \filldraw (1,3) circle (1pt);
  \filldraw (3,3) circle (1pt);
  \filldraw (4,3) circle (1pt);
  \filldraw (5,3) circle (1pt);
  \filldraw (6,3) circle (1pt);
  \draw[thick,dotted] (6,3)--(6.3,3);
  
  \draw[thick,dotted] (0.7,5)--(1,5);
  \filldraw (1,5) circle (1pt);
  \filldraw (3,5) circle (1pt);
  \filldraw (4,5) circle (1pt);
  \filldraw (5,5) circle (1pt);
  \filldraw (6,5) circle (1pt);
  \draw[thick,dotted] (6,5)--(6.3,5);
  
  \node (vv) at (0.7,2.8) {\color{red}{\footnotesize $\rho(-1)$}};
  \node (ww) at (3.3,2.8) {\color{red}{\footnotesize $\rho(2)$}};
  \node (zz) at (3.3,2.5) {\footnotesize $\theta(2)$};
  \node (xx) at (5.1,2.6) {\color{red}{\footnotesize $\rho(1)$}};
  \node (yy) at (6.4,3.2) {\color{red}{\footnotesize $\rho(3)$}};
  \node (tt) at (6.4,2.8) {\footnotesize $\theta(1)$};
  \node (uu) at (4,1) {\color{black}{\footnotesize $z$}};
  \node (rr) at (3,0) {\color{green}{\footnotesize $y$}};
  \node (ss) at (5,0) {\color{red}{\footnotesize $x$}};
  \node (qq) at (1.9,2.8) {\color{black}{\footnotesize $1$}};

  \end{tikzpicture}
  \caption{Cover associated to the subgroup $H< F(x,y,z)$; in order to encode the bijections $\rho$ and $\theta$, we identify $\mathbb Z$ with the vertices of the axis of y, i.e. the vertex corresponding to $y^n$ with $n$, for instance, $\rho(1)=3$ and so on.} \label{fig:1}
  \end{figure}
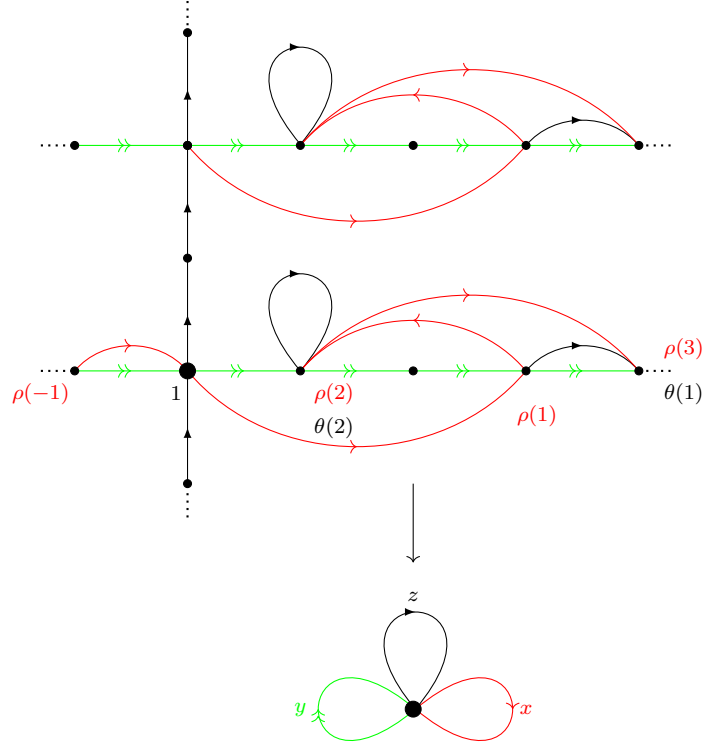

  Similarly, consider  $\tilde{H}=\tilde{H}_{\tilde{\rho}\tilde{\theta}}\leq F$ be the subgroup generated by the collection $\tilde{\mathcal{S}}=\{\tilde{v}^{l}_{j},\tilde{w}^{l}_{j}\,|\,l\in\Z, j\in\zo\}$ where
  \begin{align*}
  \bv^{l}_{j}=(\tilde{x}^{j}\tilde{y}^{-\tilde{\rho}(j)})^{\bz^{l}} \\
  \bw^{l}_{j}=(\tilde{x}^{j}\tilde{z}\tilde{y}^{-\tilde{\theta}(j)})^{\bz^{l}}
  \end{align*}
  
  For the sake of clarity we will abbreviate $v^{0}_{j}$ as $v_{j}$ and $\bv^{0}_{j}$ as $\bv_{j}$, and we will use a similar abbreviation. for $w_j^0$ and $\tilde{w}_j^0$.
  Notice that $\so$ and $\bs$ are free basis for $H$ and $\tilde H$ respectively, so any bijection $\iota$ between $\mathcal{S}$ and $\bs$ extends to an isomorphism $\iota*$ between $H$ and $\tilde{H}$.
  
  Let $G=G(\rho,\theta,\tilde{\rho},\tilde{\theta},\iota)$ be the amalgamated product of $F_{1}$ and $F_{2}$ with respect to the amalgamation induced by the isomorphism $\iota*:H \to \tilde{H}$.
  
  We will say that a minimal action of a group $G$ on a simplicial tree $T$ (without inversions) is \emph{weakly acylindrical} if it neither fixes a point nor acts 2-transitively on the boundary of $T$.
  
  \begin{thm}\label{thm:goodaction}
  The action of $G(\rho,\theta,\tilde{\rho},\tilde{\theta},\iota)$ on the tree dual to its decomposition as an amalgamated product as above is weakly acylindrical.
  \end{thm}
  
  \begin{proof} The group  $G=G(\rho,\theta,\tilde{\rho},\tilde{\theta},\iota)$ is the amalgamated product of $F_{1}=\langle x,y,x \rangle$ and $F_{2}= \langle \tilde x, \tilde y, \tilde z\rangle$ with respect to the amalgamation induced by the isomorphism $\iota*:H \to \tilde{H}$. Consider the element $h=z^{-1}\tilde{z}$. Clearly, $h$ is hyperbolic with translation length $2$; furthermore, from the definition of the subgroups $H$ and $H'$, it follows that $z$ and $\tilde z$ (and so $h$) normalize the subgroup $H=\tilde{H}$; indeed $H$ is generated by $v_j^l$ and $w_j^l$, $l\in\Z, j\in\zo$ and by definition, $(v_j^l)^z=v_j^{l+1}$ and $(w_j^l)^z=w_j^{l+1}$, so $H^{z}=H$ in $F_{1}$. Similarly, $\tilde{H}^{\tilde{z}}=\tilde{H}$ and thus $H^{h}=H$ in $G$. Since the edge in the tree whose stabliser is $H$ is in the axis of $h$, the previous fact implies that $H$ is also the pointwise stabiliser of the axis of $h$. From this we get that $h$ satisfies the property $FS(\frac{1}{2})$: if an element $g$ fixes a segment of length $1= \frac{1}{2}tl(h)$ of the axis of $h$, then $g$ belongs to $H$ and so to the pointwise stabiliser of the axis of $h$. Then, by   \cite[Lemma 2.10]{CGN1}, the element $h$ is weakly $(\frac{1}{2}+2)$-stable, see \cite[Definition 2.9.1]{CGN1}, and so is $h^n$ for any $n\in \mathbb N \setminus\{0\}$, see \cite[Remark 4.1]{CGN1}.
  Since the action of $G$ on the Bass-Serre tree is irreducible, there exists $g$ such that $\langle g,h\rangle$ also acts irreducibly and $d(A(g), A(h))<tl(h^n)$, for some $n\in \mathbb N$. From  \cite[Corollary 4.7]{CGN1}, it follows that $G$ has weakly $N$-small cancellation $m$-tuples, for all $N,m \in \mathbb N$ and from  \cite[Proposition 7.5]{CGN1} we conclude that the action of $G$ on the Bass-Serre tree is weakly acylindrical.
  \end{proof}
  \begin{cor}
  For any $(\rho,\theta,\tilde{\rho},\tilde{\theta},\iota)$ as above the group $G(\rho,\theta,\tilde{\rho},\tilde{\theta},\iota)$ has trivial positive theory and infinite dimensional second bounded cohomology. 
  \end{cor}
  
  \begin{proof}
  From Theorem \ref{thm:goodaction} we have that these groups act weakly acylindrically on a tree and so from Theorem 7.6 in \cite{CGN1} we obtain that they have trivial positive theory. On the other hand, from the definition of weakly acylindrical action and the trichotomoy provided in \cite[Theorem 1]{IPS}, we conclude that these groups have infinite dimensional second bounded cohomology.
  \end{proof}

\section{Simplicity of an uncountable subclass of groups}
  
  This section is devoted to the proof of the following result:
  \begin{prop}
  \label{p: main simple} There are $2^{\aleph_{0}}$ mutually non-isomorphic groups of the form $G(\rho,\theta,\tilde{\rho},\tilde{\theta},\iota)$ which are simple.
  \end{prop}
  
  The strategy follows the lines introduced by Ruth Camm in \cite{Camm}. The idea is to show that local conditions on elements that all together will imply the simplicity of the group, follow from local restrictions on the set of permutations. We will show that on the one hand, the set of local element conditions imply indeed simplicity of the group, and on the other one, that the local restrictions on the permutations can be consistently satisfied and extended, and that there are in fact, uncountably many parameters $\omega=(\rho, \theta, \tilde{\rho}, \tilde{\omega}, \iota)$ that satisfy them.
  
  By \emph{edge} group of an amalgamated product of two groups $G_1$ and $G_2$ we refer to the amalgamated subgroup,  and by \emph{vertex} group we mean either $G_1$ and $G_2$.
  
  The type of local conditions that we will use are the following: the normal closure of an element of \emph{length $k$} contains an element of the edge group (see Section 3.3); the normal closure of an element of the edge group of length $k$ (in the subgroup contains the generators of the vertex group (see Section 3.5). Finally, in Section 3.6, we show that we can concatenate the local conditions consistently to obtain simplicity of the whole group.

  \smallskip
  
  Let $\pb(X,Y)$ be the collection of partial bijections between two sets $X$ and $Y$.
  We write $\pb(X)=\pb(X,X)$, and denote by $\pbw(X,Y)$ the collection of all finite (i.e.\ with finite support) partial bijections between $X$ and $Y$. We define $\pbw(X)$ similarly.
  
  To prove Proposition \ref{p: main simple} we  prove  the existence of partial maps $\rho_{\text{simple}},\tilde{\rho}_{\text{simple}}\in \pb(\zo)$ and $\iota_{\text{simple}}\in \pb(\so,\bso)$ such that:
  \begin{itemize}
  \item there are countably many elements in the complement of their domains and images
  \item for any tuple $\omega=(\rho,\theta,\tilde{\rho},\tilde{\theta},\iota)$ such that $\rho, \tilde{\rho}$ and $\iota$ extend $\rho_{\text{simple}}$, $\tilde{\rho}_{\text{simple}}$, $\iota_{\text{simple}}$, respectively,  the group $G=G(\omega)$ is simple.
  \end{itemize}

  \subsection*{Some basic notation}
    
    The proof is rather technical and strongly relies on computations. To facilitate the reading, in this section we collect the notation that we will use throughout the text.
    
    \begin{defn}\label{d: cluster_of_definitions}\
    Below $X, Y, X'$ and $Y'$ denote arbitrary sets, unless otherwise specified.
    \begin{itemize}
    \item Given $f\in\pb(X)$ we let $|f|=im(f)\cup dom(f)$.
    \item Given $f\in\pb(X,X')$ and $Y\subset X$, $Y'\subset X'$ we say that $f$ is \emph{independent} from $(Y,Y')$ if $dom(f)\cap Y=im(f)\cap Y'=\emptyset$. If $f\in\pb(X)$ then we say that $f$ is \emph{independent} from $Y$ if it is independent from $(Y,Y)$.
    \item Given $X\subset\Z^{2}$, let $\mc{V}(X)$ denotes the collection of all the generators of the form $v_{j}^l=(x^{j}y^{-\rho(j)})^{z^{l}}$,  where $(j,l)\in X$. Similarly,  $\mc{W}(X)$ denotes the collection of elements $w_j^{l}=(x^{j}zy^{-\theta(j)})^{z^{l}}$, $(j,l)\in X$.
    \item $\mc{S}(X)$ denotes $\V(X) \cup \W(X)$. We define sets $\tilde{\mc{S}}(X)$, $\tilde{\V}(X)$ and $\tilde{\W}(X)$ in a similar fashion.
    \item Whenever $X$ is omitted from the notation introduced in the items above, it will be understood that $X= \mbb{Z} \times \mbb{Z}^*$
    \item Given $Z\subset\zo$ and $l\in\Z$ let $\V^{l}(Z)=\V(Z\times\{l\})$ 
    By $\V^{l}$ we refer to $\V^{l}(Z)$.
    \item Given, $k,l\in\Z$ the expression $[k,l]$ will stand for	the set $\{k,k+1,\dots l-1,l\}$. We write $\so_{N}=\so([-N,N]\times[-N,N])$; we define $\bso_{N}$ analogously.
    \item Given $(l,j)\in\Z^{2}$ let $c_{l,j}=z^{l}x^{j}$  and $\tilde{c}_{l,j}=\bz^{l}\bx^{j}$.
    \item Given $j,M\in\Z$, let $\delta(M,j)=\rho(M+j)-\rho(j)$ and $\bda(M,j)=\bro(M+j)-\bro(j)$.
    \item Given any $(M,M',l)\in\Z^{3}$, define $d_{M,M'}^{l}=v^{l}_{M}(v^{l}_{M'})^{-1}$ and, dually, define
    $\bd_{M,M'}^{l}=\bv^{l}_{M}(\bv^{l}_{M'})^{-1}$.
    \item As defined previously $F$ and $\tilde F$ denote the free groups generated by $x,y,z$ and by $\tilde{x}, \tilde{y}, \tilde{z}$, respectively. We also defined $H=\langle \mc{S}\rangle \leq F$ and $\tilde{H} = \tilde{\mc{S}} \leq \tilde{F}$.
    \end{itemize}

    \end{defn}

    As it is customary, given $X\subset G$ we let $\subg{\subg{X}}_{G}$, or simply $\subg{\subg{X}}$, stand for the normal closure of
    $X$ in $G$.
    
    Sometimes we will identify a partial map $f$ from a set $X$ to another set $Y$ with the set $\{(x,y)\mid y=f(x),\ x \in dom(f)\}$. This allows us to use expressions such as $f\subset g$ given two partial maps $f$ and $g$.

      \newcommand{\bc}[0]{\tilde{c}}
      
  \subsection*{Preliminary definitions and calculations in $G(\omega)$.}
    
    Unless otherwise specified, throughout the rest of the paper we will use $\omega=(\rho,\theta,\tilde{\rho},\tilde{\theta},\iota)$ to denote an arbitrary tuple where $\rho, \tilde{\rho}, \theta, \tilde{\theta} $ are permutations of $\mbb{Z}$ and $\iota$ is a bijection between $\mc{S}_{\rho,\theta}$ and $\tilde{\mc{S}}_{\tilde{\rho}, \tilde{\theta}}$. We shall use the notation $\mc{S}$ and $\tilde{S}$ instead of $\mc{S}_{\rho,\theta}$ and $\tilde{\mc{S}}_{\tilde{\rho}, \tilde{\theta}}$.
    
    We start with an easy observation, which we keep for later:
    \begin{lemma}
    \label{l: xy are enough} Given any tuple $\omega=(\rho,\theta,\tilde{\rho},\tilde{\theta},\iota)$, if $ \{ (s, \iota(s))\mid s \in \mc{S}\} \cap(\V\times\biw)\neq\emptyset$  then $G(\omega)=\subg{\subg{x,y,\bx,\by}}$.
    \end{lemma}
    
    \begin{proof}
    Indeed, from the fact that there is some $v_j^l \in \mc{V}$ such that $\iota(v_j^l)=\tilde{w}_{j'}^{l'}$ for some $j'\in \mbb{Z}^*, l'\in \mbb{Z}$, it follows that $\tilde z$ is in the normal closure of $x,y,\tilde x, \tilde y$. Now, since $\iota(w_j^l)$ is in the normal closure of $\tilde x, \tilde y, \tilde z$ for all $j,l$, and $\tilde z$ is in turn, in the normal closure of $x,y,\tilde x, \tilde y$, the result follows.
    \end{proof}
    
    Recall that, given $(l,j)\in\Z^{2}$ we defined $c_{l,j}=z^{l}x^{j}$  and $\tilde{c}_{l,j}=\bz^{l}\bx^{j}$.
    The following is an easy exercise:
    \begin{lemma}
    \label{l: coset enumeration}
    The set $\{c_{l,j}\}_{(j,l)\in\Z^{2}}\subset F$ is a system of representatives of
    the right cosets of $H$ in $F$. The set $\{\bc_{l,j}\}_{(j,l)\in\Z^{2}}\subset\tilde{F}$ is a system of representatives of the right cosets of $\bih$ in $\tilde{F}$.
    \end{lemma}

    Given $j,M\in\Z$, recall that $\delta(M,j)=\rho(M+j)-\rho(j)$ and $\bda(M,j)=\bro(M+j)-\bro(j)$.
    A simple calculation yields:
    \begin{align}
    \label{eq1} x^{M}v_{j}=v_{M+j}y^{\delta(M,j)}, \hspace{30pt}
    \bx^{M}\bv_{j}=\bv_{M+j}\by^{\bda(M,j)}
    \end{align}
    
    \begin{remark}
    \label{l: rewriting}If $\iota(v_{m})=\bv_{m'}$, then the corresponding defining relation of $G(w)$
    can be rewritten as
    $ x^{-m}\bx^{m'}=y^{-\rho(m)}\by^{\bro(m')}.$
    \end{remark}
    
    The following is just a straightforward calculation.
    
    \begin{lemma}
    \label{l: jump over v}In any group $G(\omega)$ we have:
    \begin{align*}
    (v_{M}^{-l})^{-1}c_{l,j}=z^{l}y^{-\delta(-j,M)}(v_{M-j})^{-1} \\
    (\bv_{M}^{-l})^{-1}\bc_{l,j}=\bz^{l}\by^{-\bda(-j,M)}(\bv_{M-j})^{-1}
    \end{align*}

    \end{lemma}
    
    Given any pair $(M,M',l)\in\Z^{3}$, recall that $d_{M,M'}^{l}=v^{l}_{M}(v^{l}_{M'})^{-1}$ and dually
    $\bd_{M,M'}^{l}=\bv^{l}_{M}(\bv^{l}_{M'})^{-1}$.
    From Lemma \ref{l: jump over v} we can easily deduce:
    \begin{lemma}
    \label{l: shift}If $(M,M',l)$ and $(j,l)$ satisfy $\delta(-j,M)=\delta(-j,M')$, then
    $$(d_{M,M'}^{-l})^{c_{l,j}}=d_{M-j,M'-j}^{0}.$$
    In particular,
    $d_{M,M'}^{x^{j}}=d_{M-j,M'-j}$
    A similar property holds in the group $\bih$ with respect to the function $\bd$.
    \end{lemma}

    \newcommand{\jj}[0]{\tilde{j}}
    
    \begin{lemma}
    \label{l: vj conjugate}Let $a_{m}$ stand for the element $x^{-m}\bx^{m}$. Suppose we are given non-trivial integers $j,\jj,M,M'\in\zo$. Then the equality  $a_{M}^{v_{j}}=a_{M'}$ holds if all the following conditions are satisfied:
    \enum{(i)}{
    \item \label{jtoj}$\iota(v_{j})=\bv_{\jj}$
    \item \label{i(M+j)} $\iota(v_{M+j})=\bv_{M+\jj}$
    \item \label{ro(M')}$\rho(M')=\delta(M,j)$
    \item \label{bro(M')} $\bro(M')=\bda(M,\jj)$
    }
    
    \end{lemma}
    \begin{proof}
    
    Using (\ref{jtoj}) and
    (\ref{eq1}) we get:
    \begin{align*}
    (x^{-M}\bx^{M})^{v_{j}}	=v_{j}^{-1}x^{-M}\bx^{M}v_{j}=\\
    =v_{j}^{-1}x^{-M}\bx^{M}\bv_{\jj}
    =y^{-\delta(M,j)}v_{M+j}^{-1}\bv_{M+\jj}\by^{\bda(M,\jj)}
    \end{align*}
    Using (\ref{i(M+j)}) for the first equality below, and then Remark \ref{l: rewriting} together with (\ref{ro(M')}) and (\ref{bro(M')}) we get:
    \begin{align*}
    y^{-\delta(M,j)}v_{M+j}^{-1}\bv_{M+\jj}\by^{\bda(M,\jj)}=y^{-\delta(M,j)}\by^{\bda(M,\jj)}=  		 	x^{-M'}\bx^{M'} .
    \end{align*}
    \end{proof}

  \subsection*{From $G\setminus\{1\}$ to $H\setminus\{1\}$ }

    Let $\mathcal{G}$ be the collection of sequences of the form   	$$u,c_{l_{0},j_{0}},\bc_{l_{1},j_{1}},\dots, c_{l_{2k},j_{2k}}, \bc_{l_{2k+1},j_{2k+1}}, $$ where $k\geq 0$,
    $u$ is a possibly empty element in the free group generated by  $\mc{S}$ and $(c_{i},j_{i})\in\Z^{2}\setminus\{(0,0)\}$ for all $1\leq 2k+1$.

    \begin{remark}\label{r: all_is_G}
    By Lemma \ref{l: coset enumeration} and the normal form theorem for amalgamated free products for any element $g\in G(w)$  there exists $u,c_{l_{0},j_{0}}, \dots, \bc_{l_{2k+1},j_{2k+1}}$ as above such that $g=uc_{l_{0},j_{0}} \dots \bc_{l_{2k+1},j_{2k+1}}$ in $G(\omega)$.
    \end{remark}

    The following result will be fundamental for the proof of Proposition \ref{p: main simple}.
    \begin{lemma}
    \label{l: hyperbolic to edge} For any sequence $\sigma\in\mathcal{G}$ of the form
    $u,c_{l_{0},j_{0}},\bc_{l_{1},j_{1}}\dots \bc_{l_{2k+1},j_{2k+1}}\in\mathcal{G}$ and for any $N>0$ there exist maps $\rho^{GH}=\rho^{GH}(N,\sigma)$, $ \bro^{GH}=\bro^{GH}(N,\sigma)\in\pbw(\Z^*)$ and $\iota^{GH}=\iota^{GH}(N,\sigma)\in\pbw(\so,\bso)$ such that:
    \begin{enumerate}[(i)]
    \item 	\label{indepdence} $\rho^{GH}$ and $\bro^{GH}$ are independent from $[-N,N]$ (in the sense of Item 2 of Definition \ref{d: cluster_of_definitions}) and $\iota^{GH}$ is independent from $(\so_{N},\bso_{N})$.
    \item 	\label{pushing into the edge group}If $\rho^{GH}\subset \rho$, $\bro^{GH}\subset\bro$ and $\iota^{GH}\subset\iota$  then in the group
    $G\tuple$ we have $\subg{\subg{g}}\cap H\neq\{1\}$, where  $g$ is the element represented by the product $uc_{l_{0},j_{0}}\bc_{l_{1},j_{1}}\dots c_{l_{2i},j_{2i}}\bc_{l_{2k+1},j_{2k+1}}$.
    \end{enumerate}
    \begin{proof}
    For sequences of integers $(M_{i})_{i=0}^{2k+2}$, $(M'_{i})_{i=0}^{2k+2}$ consider the following conditions:
    \enum{(a)}{
    	\item \label{a} $\iota(v_{M_{2i}-j_{2i}})=\bv^{-l_{2i+1}}_{M_{2i+1}}$
    	and $\iota(v_{M'_{2i}-j_{2i}})=\bv^{-l_{2i+1}}_{M'_{2i+1}}$
    	for $0\leq i\leq k$
    	  \item \label{b} $\iota(v^{-l_{2i+2}}_{M_{2i+2}})=\bv^{0}_{M_{2i+1}-j_{2i+1}}$
    	and
    	$\iota(v^{-l_{2i+2}}_{M'_{2i+2}})=\bv^{0}_{M'_{2i+1}-j_{2i+1}}$
    	for $0\leq i\leq k$.
    	\item \label{c}$\rho(M_{i}-j_{i})=\rho(M_{i})-j_{i}$
    	and $\rho(M'_{i}-j_{i})=\rho(M'_{i})-j_{i}$
    	for $0\leq i\leq 2k+2$
    	\item \label{d} $\bro(M_{i}-j_{i})=\bro(M_{i})-j_{i}$
    	and $\bro(M'_{i}-j_{i})=\bro(M'_{i})-j_{i}$
    	for $0\leq i\leq 2k+2$
    	\item \label{e}
    	$v_{M_0}^{-l_0}$ and $v_{M_0'}^{-l_0}$ do not appear in the alphabet of the word $u$
    	\item \label{f} $\max_{0\leq i\leq 2k+1}|j_{i}|< \min_{0\leq i<i'\leq 2k+2}|M_{i}-M_{i'}|,\min_{0\leq i<i'\leq 2k+2}|M'_{i}-M'_{i'}|$ 
    	
    	\item \label{g} $\max_{0\leq i\leq 2k+1}|j_{i}|<\min_{0\leq i,i'\leq 2k+2}\{|M_{i}-M'_{i'}|\}$. In particular, $M_i\ne M_i'$ for all $i=0, \dots, 2k+2$. 
    }
    We will need the following
    \begin{claim}\label{c: claim}
    	If $(M_{i})_{i=0}^{2k+2}$, $(M'_{i})_{i=0}^{2k+2}$, and $\rho, \tilde{\rho}, \iota$ are integer sequences and  maps such that the conditions (\ref{a})-(\ref{d}) above hold.  Then for $-1\leq i\leq k$ we have:
    	\begin{align*}
    		(d^{-l_{0}}_{M_{0},M'_{0}})^{g_{i}}=d_{M_{2i+2},M'_{2i+2}}^{-l_{2i+2}}
    	\end{align*}
    	where $g_{i}$ is defined as $ c_{l_{0},j_{0}}\bc_{l_{1},j_{1}}\dots c_{l_{2i},j_{2i}}\bc_{l_{2i+1},j_{2i+1}}$ if $i\geq 0$ and $g_{-1}$ is set to be $1$.
    \end{claim}
    Indeed, we proceed by descending induction. On the one hand, by Condition $(c)$ above we have $$\delta(-j_{2i} ,M_{2i}) = \rho(-j_{2i} + M_{2i}) - \rho(M_{2i}) = 0 = \rho(-j_{2i} + M_{2i}') - \rho(M_{2i}')= \delta(-j_{2i}, M_{2i}'),$$ and so we can use  Lemma \ref{l: shift} and  Condition (\ref{a}) to obtain:
    \begin{align*}
    	(d^{-l_{2i}}_{M_{2i},M'_{2i}})^{c_{l_{2i},j_{2i}}}=d^{0}_{M_{2i}-j_{2i},M'_{2i}-j_{2i}}=\bd^{-l_{2i+1}}_{M_{2i+1},M'_{2i+1}}.
    \end{align*}
    A similar argument involving condition (\ref{b}) and (\ref{d}) yields:
    \begin{align*}
    	(d^{-l_{2i}}_{M_{2i},M'_{2i}})^{c_{l_{2i},j_{2i}}\bc_{l_{2i+1},j_{2i+1}}}=(\bd^{-l_{2i+1}}_{M_{2i+1},M'_{2i+1}})^{\bc_{l_{2i+1},j_{2i+1}}}=\\
    	=\bd^{0}_{M_{2i+1}-j_{2i+1},M'_{2i+1}-j_{2i+1}} =d^{-l_{2i+2}}_{M_{2i+2},M'_{2i+2}}
    \end{align*}
    This completes the proof of Claim \ref{c: claim}.
    
    We shall now use Claim \ref{c: claim} in order to prove Item (ii) of Lemma \ref{l: hyperbolic to edge}. Suppose that we are given integer sequences  $(M_{i})_{i=0}^{2k+2}$, $(M'_{i})_{i=0}^{2k+2}$, as well as $\rho, \tilde{\rho}, \iota$ satisfying the conditions $(a)$ through $(g)$ above,  and consider the element $e=(d^{-l_{0}}_{M_{0},M'_{0}})^{u^{-1}}$. The  element $e^{g}e^{-1}= g^{-1} g^{e^{-1}}$ belongs to the normal closure of $g$. According to  Claim \ref{c: claim} and since $u g_k = g$, the  element $e^g e^{-1}$ is equal to $e'=d^{l_{2k+2}}_{M_{2k+2},M'_{2k+2}}e^{-1}\in H$.
    By Condition (\ref{e}) neither $v^{-l_{0}}_{M_{0}}$ nor
    $v^{-l_{0}}_{M'_{0}}$ appear in $u$ and by Condition (\ref{f}), $M_i \ne M_i'$ and so it follows that $e'$
    is non-trivial and we are done.
    
    Now, it is easy to find tuples $(M_{i})_{i}$ and $(M'_{i})_{i}$ satisfying conditions $(\ref{e}), (\ref{f})$ and $(\ref{g})$.
    For fixed values of $\rho(M_{i}),\bro(M_{i}),\rho(M'_{i}),\bro(M'_{i})$ conditions (\ref{a}) to (\ref{d})
    can be expressed in terms of $\iota,\rho,\bro$ containing some partial maps $\iota^{GH}$, $\rho^{GH}$, and $\bro^{GH}$.
    The fact that this choice can be made in a way compatible with the bijectivity of $\rho$ follows easily from conditions (\ref{f}) and (\ref{g}). As for the bijectivity of $\iota$, we need to show that the elements on the left hand-side of Conditions (\ref{a}) and (\ref{b})
    $$
    \left\{ v_{M_{2i}-j_{2i}}, v_{M'_{2i}-j_{2i}}, v^{-l_{2i+2}}_{M_{2i+2}}, v^{-l_{2i+2}}_{M'_{2i+2}} \mid i = 0, \dots, k\right\},
    $$
    as well as their images
    $$
    \left\{ \bv^{-l_{2i+1}}_{M_{2i+1}},\bv^{-l_{2i+1}}_{M'_{2i+1}}, \bv_{M_{2i+1}-j_{2i+1}}, \bv_{M'_{2i+1}-j_{2i+1}} \mid  i= 0, \dots, k\right\},
    $$
    are pairwise different.
    
    Let us show that  $v_{M_{2i'}-j_{2i'}}\ne v^{-l_{2i+2}}_{M_{2i+2}}$, for all $0\leq i,i'\leq k$.
    From condition (\ref{f}), we have that $$m_{j}=\max_{0\leq i\leq 2k+1}|j_{i}|< \min_{0\leq i<i'\leq 2k+2}|M_{i}-M_{i'}|$$ and so if $M_{2i'}-j_{2i'}=M_{2i+2}$, then
    $i'=i+1$ and $j_{2i+2}=0$. Furthermore, if
    $j_{2i+2}=0$, then $l_{2i+2}\neq 0$ and hence we have that $v_{M_{2i'}-j_{2i'}}\ne v^{-l_{2i+2}}_{M_{2i+2}}$. On the other hand, if $j_{2i+2}\neq 0$, then the inequality is immediate. Using symmetric arguments, one shows that all elements are pairwise different.
    
    In order to show that the target elements are also pairwise different, that is $\bv^{-l_{2i+1}}_{M_{2i+1}} \ne \bv_{M_{2i'+1}-j_{2i'+1}}$ for all $0\leq i,i'\leq k$, one uses analogous arguments, using Condition (\ref{g}) instead of (\ref{f}).
    
    Independence from $[-N,N]$ and  $(\so_{N},\bso_{N})$ simply follows by taking $M_{i}$, $M'_{i}$ and their images by $\rho$ and $\bro$ large enough.

    \end{proof}
    
    \end{lemma}
    
  \subsection*{From $H\setminus\{1\}$ to $V^{0}\setminus\{1\}$ }

    \begin{lemma}
    \label{l: vertical}For any $N>0$ there exist $\iota^{vert}=\iota^{vert}(N)\in\pbw(\so,\bso)$ and $X=X^{vert}(N)\subset\Z$
    such that:
    \enum{(i)}{
    \item 	\label{indepdence2} $\iota^{vert}$ is independent from $(\so_{N},\bso_{N})$.
    \item \label{tamenesss}$\{(v_{j},\bv_{j})\}_{j\in X}\subset\iota^{vert}$
    \item 	\label{pushing into vertical}If  $\iota^{vert}\subset\iota$, then there exists  $h\in G(\omega)$ such that
    $h^{-1}\so_{N}h\subseteq \langle \mc{V}(X\times\{0\}) \rangle$.

    }
    \end{lemma}
    \begin{proof}
    Take for example $X=[2N,2N+4N(2N+1)-1]$. Define $B_N$ to be $([-N,N]\setminus\{0\})\times[-N,N]$ and choose any bijection $\mu:B_{N} \times\{0,1\}\to X$  and let:
    \begin{align*}
    \iota^{vert}=\{(v_{j}^{l+N+1},\tilde{v}_{\mu(j,l,0)}^{1}),(w_{j}^{l+N+1}, \tilde{v}_{\mu(j,l,1)}^{1})\,|\,(j,l)\in B_{N}\}\cup\{(v_{j},\bv_{j})\,|\,j\in X\}.
    \end{align*}
    
    The relation $\iota^{vert}$ is clearly a partial bijection independent from $(\so_{N},\bso_{N})$. As $h$ we take the element $z^{N+1}\bz^{-1}$. We content ourselves with checking condition (\ref{pushing into vertical})
    on an element of the form $v_{j}^{l}$ with $(j,l)\in B_{N}$. Assume $\iota^{vert}\subset\iota$. Then:
    \begin{align*}
    (v_{j}^{l})^{h}=(v_{j}^{l})^{z^{N+1}\bz^{-1}}=(v_{j}^{l+N+1})^{\bz^{-1}}=(\bv^{1}_{\mu(j,l,0)})^{\bz^{-1}}=\bv_{\mu(j,l,0)}=v_{\mu(j,l,0)}.
    \end{align*}
    \end{proof}
  \subsection*{Killing $\subg{\bx,\by}$ }
    
    Given $X\subseteq \mbb{Z}^*$, let $\mathcal{H}(X)$ be the collection of non-empty sequences of the form
    $u=v_{j_{0}}^{\epsilon_{0}}v_{j_{1}}^{\epsilon_{1}}\dots v_{j_{k}}^{\epsilon_{k}}$ ($v_{j}$ is regarded as a letter here) where $\epsilon_{k}\in\{1,-1\}$, $j_{k}\in X$, and   where $k\geq 1$, $j_{i}=j_{i+1}$ implies $\epsilon_{i}=\epsilon_{i+1}$. If $X =\mbb{Z}^*$ then we denote $\mc{H}(X)$ simply by $\mc{H}$.  
    Let us restate an argument from \cite{Camm}.
    \begin{lemma}
    \label{l: stable vertical} Suppose we are given $N>0$ and
    $u=v_{j_{0}}^{\epsilon_{0}}v_{j_{1}}^{\epsilon_{1}}\dots v_{j_{k}}^{\epsilon_{k}}\in\mathcal{H}$ such that $\max_{0\leq i\leq k}|j_{i}|<N$.
    
    Then there exist $\rho^{x\bx}=\rho^{x\bx}(N,u)$, $\bro^{x\bx}=\bro^{x\bx}(N,u)\in\pb
    (\Z)$, $\iota^{x\bx}=\iota^{x\bx}(N,u)\in\pb(\so,\bso)$ such that:
    \enum{(I)}{
    \item \label{independence3} The maps $\rho^{x\bx}$ and $\bro^{x\bx}$ are independent from $[-N,N]$, and $\iota^{x\bx}$ is independent from $(\so_{N},\bso_{N})$.
    \item If $Id_{\{j_{i}\}_{i=0}^{k}}\cup \rho^{x\bx}\subset\rho$, $Id_{\{j_{i}\}_{i=0}^{k}} \cup\bro^{x\bx}\subset\bro$,
    and $\{(v_{j_{i}},\bv_{j_{i}})\}_{1\leq i\leq k}\cup\iota^{x\bx}\subset\iota$, then in $G(\omega)$ we have $\bx, \by\in\subg{\subg{u}}$.
    }
    
    \end{lemma}
    \begin{proof}
    \newcommand{\ros}[0]{\rho^{x\bx}}
    \newcommand{\bros}[0]{\bro^{x\bx}}
    \newcommand{\ios}[0]{\iota^{x\bx}}
    Let $M_{i}=2(i+1)N$ ($i=0,\dots, k$), $C=4kN$. We let $\ros,\bros,\ios$ be defined by the conditions:
    \enum{(i)}{
    \item $\rho(M_{i})=\bro(M_{i})=M_{i}$ for all $0\leq i \leq k$,
    \item \label{multiple1}$\rho(C)=\bro(C)=C$, $\rho(C+1)=C+1$, $\bro(C+1)=C+2$,
    \item \label{multiple2} $\rho(2C)=\bro(2C)=2C$, $\rho(2C+1)=2C+1$, $\bro(2C+1)=2C+3$
    \item \label{multiple11}$\rho(3C)=3C$, $\rho(3C+2)=3C+1$,
    \item \label{multiple21} $\rho(4C)=4C$, $\rho(4C+3)=4C+1$
    \item \label{multiple stable}$\iota(v_{k})=\bv_{k}$ for $k\in\{C,C+1,2C,2C+1\}$.
    
    \item If $\epsilon_{i}=1$, then:
    \enum{(a)}{
    	\item $\rho(M_{i}+j_{i})=M_{i+1}+j_{i}$
    	\item $\bro(M_{i}+j_{i})=M_{i+1}+j_{i}$
    	\item  $\iota(v_{M_{i}+j_{i}})=\bv_{M_{i}+j_{i}}$
    }
    \item If $\epsilon_{i}=-1$, then:
    \enum{(a)}{
    	\item $\rho(M_{i+1}+j_{i})=M_{i}+j_{i}$
    	\item $\bro(M_{i+1}+j_{i})=M_{i}+j_{i}$
    	\item $\iota(v_{M_{i+1}+j_{i}})=\bv_{M_{i}+j_{i}}$  			}
    }
    Notice that this implies (using $\rho(j_{i})=\bro(j_{i})=j_{i}$ as a consequence of Item II in the statement of the lemma) that
    $\delta(M_{i},j_{i})=M_{i+1}=\rho(M_{i+1})=\bro(M_{i+1})$ in case $\epsilon_{i}=1$ and $\delta(M_{i+1},j_{i})=M_{i}=\rho(M_{i})=\bro(M_{i})$ in case $\epsilon_{i}=-1$.

    By virtue of Lemma \ref{l: vj conjugate} together with  $\rho(j_{i})=\bro(j_{i})=j_{i}$  we have $(x^{-M_{i}}\bx^{M_{i}})^{v_{j_{i}}^{\epsilon_{i}}}=x^{-M_{i+1}}\bx^{M_{i+1}}$ 
    for all $0\leq i\leq k-1$, from which it follows that $(x^{-M_{0}}\bx^{M_{0}})^{u}=x^{-M_{k}}\bx^{M_{k}}.$ And thus $x^{-m}\bx^m= x^{-(M_k-M_0)}\bx^{M_k-M_0}=[x^{-M_0} \bx^{M_0}, u]^{\bx^{-M_0}}\in\subg{\subg{u}}$, where $m=2kN$ since $M_i=2(i+1)N$.
    
    It follows that
    $x^{C}=\bx^{C}$ mod $\subg{\subg{u}}$.
    Now, the equalities $\rho(C)=\bro(C)=C$ and
    $\iota(v_{C})=\bv_{C}$ imply that also $y^{C}=\by^{C}$ mod $\langle \langle u \rangle \rangle$, since $x^C y^{-\rho(C)} = \tilde{x}^C \tilde{y}^{-\tilde{\rho}(C)}$. Applying both to the equality $x^{C+1}y^{-C-1}=\bx^{C+1}\by^{-C-2}$   yields
    $xy^{-1}=\bx\by^{-2}$ mod $\subg{\subg{u}}$. An analogous argument using (\ref{multiple2}) yields
    $xy^{-1}=\bx\by^{-3}$ mod $\subg{\subg{u}}$. It follows that $\by\in \subg{\subg{u}}$. In a similar fashion clauses
    (\ref{multiple11}) and (\ref{multiple21}) yield $\bx\in \subg{\subg{u}}$.
    
    Define now $M'_{i}=M_{i}+j_{i}$,
    $M''_{i}=M_{i+1}+j_{i}$, and let $\mathcal{M}=\{M_{i}\}_{0\leq i\leq k+1}$
    $\mathcal{M}'=\{M'_{i}\}_{0\leq i\leq k}$
    and $\mathcal{M}''=\{M''_{i}\}_{0\leq i\leq k}$.
    From the facts that $\max_{0\leq i\leq k}|j_{i}| < N$ (by hypothesis)   and $j_{i}\neq 0$ for all $i$  it follows that  the integers in $\mc{M}$ are pairwise distinct, and similarly for $\mc{M}'$ and $\mc{M}''$.  Moreover $\mathcal{M}\cap(\mathcal{M}'\cup\mathcal{M}'')=\emptyset$.
    
    On the other hand, the only way we could possibly have $M'_{i'}=M''_{i}$ is if $i'=i+1$, $\epsilon_{i}=1$, $\epsilon_{i'}=-1$
    and $j_{i}=j_{i'}$, which cannot occur by assumption. Since $C$ is bigger than any of the $M_{i}$ or $M'_{i}$, it easily follows that the list of conditions above is consistent, giving partial bijections $\ros,\bros,\ios$ as in the statement.
    That the latter are independent from $[-N,N]$ and $(\so_{N},\bso_{N})$ is clear from the definition of $M_{i}$.
    \end{proof}
    
    \begin{remark}
    The condition  $\{(v_{j_{i}},\bv_{j_{i}})\}_{1\leq i\leq k}\cup\iota^{x\bx}\subset\iota$ in Item II of Lemma \ref{l: stable vertical}   could be replaced by the weaker requirement that $\iota(v_{j_{i}})\in\biv$
    and the condition that $\rho(j_{i})=\bro(j_{i})=j_{i}$ could be removed altogether, at the expense of complicating the calculations.
    \end{remark}

  \subsection*{Proof of Proposition \ref{p: main simple}}
    
    We construct a sequence of positive integers  $N_{0}<N_{1}<\dots <N_{k}<N_{k+1}\dots$ and an increasing
    sequences of partial bijections
    $$(\rho_{k})_{k\geq 0},(\bro_{k})_{k\geq 0}\subset\pb([-N_{k},N_{k}]),\quad (\iota_{k})_{k\geq 0}\subset\pb(\so_{N_{k}},\bso_{N_{k}}),$$
    a sequence of disjoint sets
    $(X_{k})_{k\geq 1}$ and a sequence of elements $(u^{*}_{k})_{k\geq 1}\subset\mathcal{H}$ as follows.
    
    To begin with, let $\rho_{0}=\bro_{0}=\emptyset$, $\iota_{0}=\{(v_{1},\bw_{1}),(w_{1},\bv_{1})\}$ and $N_{0}=2$.
    Let $(s_{n})_{n\geq 0}$ be an enumeration of the set $\mathcal{G}$ and $(w_{n})_{n\geq 0}$ be an enumeration of $\mathcal{H}$.

    Assume now that $k>1$ and  $\rho_{k-1},\bro_{k-1},\iota_{k-1}$, $X_{k-1}$, $N_{k-1}$ have already been constructed.
    
    Consider the partial maps given by Lemma  \ref{l: hyperbolic to edge}, $\rho^{GH}=\rho^{GH}(N_{k-1},s_{k-1})$,  $\bro^{GH}=\bro^{GH}(N_{k-1},s_{k-1})$, and $\iota^{GH}=\iota^{GH}(N_{k-1},s_{k-1})$.
    Let $N'>N_{k-1}$  be such that
    $\rho^{GH},\bro^{GH}\in\pbw([-N',N'])$ and $\iota^{GH}\in\pbw(\so_{N'},\bso_{N'})$.
    Furthermore, let $\iota^{vert}(N')$ and
    $X_{k}:=X^{vert}(N')$ be the output from Lemma \ref{l: vertical} and set
    $\bro^{vert}=Id_{X_{k}}$.
    Let $u^{*}_{k}=w_{l}$ where $l$ is the minimum index $l$ such that
    $$w_{l}\in \cup_{j\leq k}\mathcal{H}(X_{j})\setminus\left\{u_{j}^*\mid 0 \leq j< k\right\},$$ 
    Finally, consider the maps $\rho^{x\bx} =\rho^{x\bx}(N',u^{*}_{k})$, $\rho^{x\bx}=\rho^{x\bx}(N',u^{*}_{k})$, $\iota^{x\bx}=\iota^{x\bx}(N',u^{*}_{k})$ given by Lemma \ref{l: stable vertical} and let
    \begin{align*}
    \rho_{k}= &\rho_{k-1}\cup\rho^{GH}\cup\rho^{vert}\cup\rho^{x\tilde{x}} \\
    \bro_{k}= &\bro_{k-1}\cup\bro^{GH}\cup\bro^{vert}\cup\bro^{x\tilde{x}} \\
    \iota_{k}= &\iota_{k-1}\cup\iota^{GH}\cup\iota^{vert}\cup\iota^{x\tilde{x}}
    \end{align*}
    Choose $N_{k}>N'$ such that $\rho_{k},\bro_{k}\in\pb([-N_{k},N_{k}])$ and $\iota_{k}\in\pb(\so_{N_{k}},\bso_{N_{k}})$ hold. 
    
    To conclude, let 
    \begin{align*}
    \rho_{\text{simple}}= &\bigcup_{k\geq 0}\rho_{k},\;\;\;\bro_{\text{simple}}= \bigcup_{k\geq 0}\bro_{k},\;\;\; \iota_{\text{simple}}= \bigcup_{k>0}\iota_{k}
    \end{align*}
    
    \begin{claim}
    $\rho_{\textit{simple}}\subset\rho,\,\,\bro_{\textit{simple}}\subset\bro,\,\,\iota_{\textit{simple}}\subset\iota$ implies that $G\tuple$ is simple.
    \end{claim}
    \begin{proof}
    Let $G=G\tuple$ as above and $g$ a non-trivial element of $G$ and let $N_g=\subg{\subg{g}}$. Let us say that $g$ is represented by $s_{k_{0}} \in \mc{G}$ for some $k_{0}>0$ (see Remark \ref{r: all_is_G}).
    
    Lemma \ref{l: hyperbolic to edge} implies the existence of some element $h\in (H\cap N_g)\setminus\{1\}$.
    Now pick some $k_{1}>0$ such that  $h\in\subg{\so_{N_{k_{1}}}}$. By Lemma \ref{l: vertical} we must have $h'\in N_g$ for some
    $h'\in\mathcal{H}(X_{k_{1}})\setminus\{1\}$.
    
    Now, there must be some $k_{2}\geq k_{1}$ such that
    $h'=u^{*}_{k_{2}}$
    Since $Id_{X_{k_{1}}}\subset\rho,\bro$,
    and $(v_{j},\bv_{j})_{j\in X_{k_{1}}}\subset\iota$, by the third step in the construction of $\rho_{k_{2}},\bro_{k_{2}},\iota_{k_2}$ and Lemma \ref{l: stable vertical} we have
    $\bx,\by\in N_g$. It follows that also $x,y\in N$. The fact that $(v_1,\tilde{w}_{1}), (w, \tilde{v}) \in\iota_{\text{simple}}$ then implies that $z,\bz\in N$.
    \end{proof}

\end{document}